\documentclass[final,1p,times]{elsarticle}
\usepackage{amsthm,amssymb,url,hyperref,amsfonts,amsmath}
\usepackage{graphicx,tikz-cd}

\usepackage{xcolor}
\usepackage{youngtab}
\usepackage{ytableau}
\usepackage[english]{babel}
\usepackage[utf8]{inputenc}
\usepackage{amsmath}
\usepackage{amssymb}
\usepackage{amsfonts}
\usepackage{amsthm}
\usepackage{mathrsfs}
\usepackage[all]{xy}
\usepackage{url}
\usepackage{indent first}
\usepackage[labelfont=bf,labelsep=period,justification=raggedright]{caption}
\usepackage[english]{babel}
\usepackage[utf8]{inputenc}
\usepackage{hyperref}
\usepackage[colorinlistoftodos]{todonotes}
\usepackage{tkz-fct}
\usepackage{tikz}
\usepackage{pgfplots}
\usepgfplotslibrary{polar}

\hypersetup{
	colorlinks=true, allcolors=teal,linkcolor=teal, citecolor=teal,backref=page
}

\newcommand{\la}{\lambda}

\newcommand{\CC}{\mathbb{C}}

\newcommand{\PP}{\mathbb{P}}

\newcommand{\RR}{\mathbb{R}}

\newcommand{\ZZ}{\mathbb{Z}}

\newcommand{\mcA}{\mathcal{A}}

\newcommand{\mcC}{\mathcal{C}}

\newcommand{\mcH}{\mathcal{H}}

\newcommand{\mcM}{\mathcal{M}}

\makeatletter
\newcommand{\skipitems}[1]{%
	\addtocounter{\@enumctr}{#1}%
}
\makeatother

\newtheorem{thm}{Theorem}
\newtheorem{lemma}[thm]{Lemma}

\newtheorem{conj}[thm]{Conjecture}
\newtheorem{prop}[thm]{Proposition}

\theoremstyle{definition}

\newtheorem{ex}[thm]{Example}

\usepackage{pgfplots}
\pgfplotsset{compat=newest}

\usepackage{amsthm,amssymb, url, hyperref}
\usepackage{graphicx}

\usepackage{xcolor}

\newcommand{\ccc}{\mathbf{c}}
\newcommand{\KWP}{(\CC^*)^{2^g} \times \mathbb{W} \PP^{3g-1}}
\newcommand{\HC}{\mathcal{H}_\mathcal{C}}
\newcommand{\uu}{{\bf u}}
\newcommand{\vv}{{\bf v}}
\newcommand{\ww}{{\bf w}}
\newcommand{\HM}{\mathcal{H}_\mathcal{C}^{M}}

\newcommand{\norm}[1]{\left\lVert#1\right\rVert}
\newcommand{\Ig}{\mathfrak{I}_g}

\date{ }
 
\begin{document}

\begin{frontmatter}
	
\title{Hirota Varieties and Rational Nodal Curves}

\author[1]{Claudia Fevola}
\ead{claudia.fevola@mis.mpg.de}

\author[2]{Yelena Mandelshtam}
\ead{yelenam@berkeley.edu}

\address[1]{MPI-MiS, Leipzig, Germany}
\address[2]{UC Berkeley, United States of America}

\begin{abstract}
	\noindent 
The Hirota variety parameterizes solutions to the KP equation arising from a degenerate Riemann theta function. In this work, we study in detail the Hirota variety arising from a rational nodal curve. Of particular interest is the irreducible subvariety defined as the image of a parameterization map, we call this the main component. Proving that this is an irreducible component of the Hirota variety corresponds to solving a weak Schottky problem for rational nodal curves. We solve this problem up to genus nine using computational tools.
\end{abstract}

\end{frontmatter}

\label{intro}

The \textit{degenerate Riemann theta function} associated to an irreducible rational nodal curve $X$ of genus $g$ is a finite sum of exponentials 

\begin{equation}\label{eq:degenerateTheta}
\theta_{\mathcal{C}}({\bf z}) \,\,=\,\, a_1 \exp[\ccc_1^T {\bf z}] \,\,+\,\, a_2 \exp[\ccc_2^T {\bf z}] \,\,+\,\, \dots \,\,+\,\, a_{2^g} \exp[\ccc_{2^g}^T {\bf z}]
\end{equation}
supported on the configuration of points $\mathcal{C} = \{0,1\}^g$ in the integer lattice $\ZZ^g$. Each point $\ccc_i = (\ccc_{i1},\ccc_{i2},\dots,\ccc_{ig})\in {\cal C}$ determines a linear form $\ccc_i^T {\bf z} = \sum_{j=i}^g c_{ij} z_j$, where ${\bf z}=(z_1,z_2,\dots,z_g)$ is a vector of complex unknowns, while the coefficients ${\bf a}~=~(a_1,a_2,\dots,a_{2^g})$ play the role of coordinates on the algebraic torus $(\CC^*)^{2^g}$.
The classical Riemann theta function associated to a smooth algebraic curve of genus $g$ is a complex analytic function expressed by an infinite sum of exponentials supported on the whole integer lattice, namely 
\begin{equation}\label{eq:thetaFunction}
\theta \,\,=\,\, \theta({\bf z}\,|\, B)\,\, =\,\,\sum_{{\bf c}\in \ZZ^g} \exp\left[ \frac{1}{2} \ccc^T B \ccc + \ccc^T {\bf z}\right],
\end{equation}
where $B$ is the \textit{Riemann matrix} associated to the curve. This is a symmetric $g\times g$ matrix that we normalize to have negative definite real part. 
The analogue of this matrix for a rational nodal curve is given by the \textit{tropical Riemann matrix}. To compute it, one considers the dual graph $\Gamma$ \cite{chan2017lectures} corresponding to $X$ and associates a matrix to it by looking at the generators of the first homology group of the graph $\Gamma$. The dual graph of the curve $X$ consists of one unique vertex with precisely $g$ cycles attached to it, one for each node, and the tropical Riemann matrix is the identity matrix of size $g$.

The \textit{Hirota variety} $\HC$ associated to the curve $X$, first introduced in \cite{AFMS} for curves with at worst nodal singularities, lives in the parameter space given by the product of the algebraic torus $(\CC^*)^{2^g}$ with coordinates ${\bf a} = (a_1,\dots,a_{2^g})$, and the weighted projective space $\mathbb{W}\PP^{3g-1}$ whose coordinates are the vectors ${\bf u}  = (u_1,\dots,u_g)$, ${\bf v}  = (v_1,\dots,v_g)$, and ${\bf w }  = (w_1,\dots,w_g)$ with weights $1$,$2$, and $3$ for $u_i, v_i$, and $w_i$, respectively. It parameterizes the points ${({\bf a},({\bf u},{\bf v},{\bf w})) \in (\CC^*)^{2^g}\times\mathbb{W}\PP^{3g-1}}$ such that the $\tau$-function $\tau(x,y,t) = \theta_{\mcC}(\uu x+\vv y +\ww t)$ satisfies the \textit{Hirota's bilinear equation}
\begin{equation}\label{eq:hirotabilinearform}
\tau \tau_{xxxx}\,\,-\,\,4\tau_{xxx}\tau_x\,\, +\,\, 3\tau_{xx}^2\,\,+\,\,4\tau_x\tau_t \,\,-\,\,4\tau\tau_{xt}\,\,+\,\,3\tau \tau_{yy}\,\, -\,\, 3\tau_y^2 \,\,=\,\, 0.
\end{equation}
This provides a sufficient condition for the function 
\begin{equation}\label{eq:solutions}
p(x,y,t)\,= \,2\frac{\partial^2}{\partial x^2} \log \tau (x,y,t)    
\end{equation}
to be a solution for the \textit{Kadomtsev-Petviashvili (KP) equation}. This is a nonlinear partial differential equation describing the motion of water waves given by
\begin{equation}\label{eq:KP}
\frac{\partial}{\partial x} (4 p_t - 6pp_x -p_{xxx})\,=\,3p_{yy},
\end{equation}
where the unknown function $p = p(x,y,t)$ depends on two space variables $x$ and $y$, and on the time variable $t$. Furthermore, points in the Hirota variety $\HC$ correspond to a special class of solutions to the KP equation, namely \textit{soliton} solutions, \cite{AbeGri2018, Kodama_book}.

These are solutions of type \eqref{eq:solutions} with the $\tau$-function having a particular form. To define it, one fixes integers $k<n$ and a vector of parameters $\kappa=(\kappa_1,\dots,\kappa_n)\in \RR^n$, then
\begin{equation}\label{eq:tauSoliton}
\tau(x,y,t) \,\,=\,\, \sum_{I\in\binom{[n]}{k}}p_I\prod_{\substack{i,j\in I\\i<j}} (\kappa_i-\kappa_j) \cdot\exp\biggl[ x\cdot\sum_{i\in I}\kappa_i + y\cdot\sum_{i\in I}\kappa_i^2 + t\cdot\sum_{i\in I}\kappa_i^3\biggr],
\end{equation}
where $I$ runs over all possible $k$-tuples of elements in the set $[n] = \{1,2,\dots,n\}$, and the $p_I$ are Plücker coordinates of a point in the Grassmannian Gr$(k,n)$, \cite{Kodama_book,sato1981soliton}. We call \eqref{eq:tauSoliton} a $(k,n)$-soliton.

Finding solutions to the KP equation is related to the \emph{Schottky problem}, in that a theta function satisfies the KP equation when the corresponding abelian variety is the Jacobian of a curve \cite{krichever2013soliton,shiota1986characterization}. More precisely, it concerns the Torelli map $J: \mathcal{M}_g \to \mathcal{A}_g$, where $\mathcal{M}_g$ is the moduli space of curves of genus $g$, and $\mathcal{A}_g$ is the moduli space of abelian varieties of dimension $g$. The Schottky problem is to find the defining equations for the locus of Jacobians, defined as the closure of $J(\mathcal{M}_g)$ in $\mathcal{A}_g$. 
The \emph{weak Schottky problem} is to find an ideal whose zero locus contains the locus of Jacobians as an irreducible component. In genus 3, the Schottky problem is trivial, in the sense that every polarized abelian variety is the Jacobian of a curve. For genus $g\geq4$, one has a proper inclusion $J(\mathcal{M}_g)\subset\mathcal{A}_g$. For genus larger than $5$, the Schottky problem has proven to be difficult, with only partial results for genus 5 and weak solutions in higher genera. Many approaches to the Schottky problem have been developed (e.g., \cite{arbarello1984set,igusa1982irreducibility}), and an explicit solution to the weak Schottky problem was provided in \cite{farkas2021explicit} and investigated numerically for genus 5 in \cite{agostini2021computing}. The proof of Theorem $\ref{thm:irredcomp}$ in our paper relies on a solution to the weak Schottky problem for nodal curves, which we provide for $g \leq 9$.

This article is organized as follows. In Section \ref{sec:maincomp}, we introduce the Hirota variety of a rational nodal curve and discuss its main component. We prove for genus $g\leq 9$ that this is an irreducible component of dimension $3g$ and we explain how this relates to the Schottky problem. Section 
\ref{sec:combinatorics} studies the equations of the main component of the variety $\HC$ and how they relate to the combinatorics of the cube. 

\section{The Main Component of the Hirota Variety}\label{sec:maincomp}
Let $X$ be a rational nodal curve of genus $g$. The degenerate theta function \eqref{eq:degenerateTheta} associated to $X$ is supported on the vertices of the $g$-dimensional cube $\mcC$ whose vertices are all binary $g$-dimensional vectors. The Hirota variety $\HC$ lives in the space $(\CC^*)^{2^g}\times\mathbb{W}\PP^{3g-1}$ with coordinate ring $\CC[{\bf a}^{\pm 1}, {\bf u}, {\bf v}, {\bf w}]$, where $\deg(u_i)=1,\deg(v_i)=2,$ and $\deg(w_i)=3$ for $i=1,2,\dots,g$. 

We want to investigate the subvariety denoted by $\HC^{I}$ of the Hirota variety $\mathcal{H}_\mathcal{C}$, where the superscript $I$ stands for `invertible'. We define $\HC^I$ to be the Zariski closure of the set 
\begin{equation}
\bigl\{ \, (\mathbf{a, (u, v, w)}) \in \HC \,\,:\,\, \bf{u}\neq \bf{0}\, \bigr\}.
\end{equation}
This subvariety also contains a meaningful irreducible subvariety, which we call the \textit{main component}. 
To rigorously define it, consider the map from the affine space $\CC^{3g}$ with coordinates
$(\lambda_1,\dots,\lambda_g,\kappa_1,\kappa_2,\dots, \kappa_{2g})$ into the ambient space of the Hirota variety $\HC$ given by 
\begin{align}\label{eq:paramMap}
\phi \;:\; \CC^{3g} \quad & \dashrightarrow \,\KWP\\
(\lambda_1,
\dots,\lambda_g,
\kappa_1,\kappa_2,
\dots,\kappa_{2g}) & \longmapsto
\,(a_{\ccc_1}, a_{\ccc_2},
\dots,a_{\ccc_{2^g}},
{\bf u}, {\bf v}, {\bf w}),\nonumber
\end{align}
where the coordinates ${\bf a} = (a_{\ccc_1}, a_{\ccc_2},
\dots,a_{\ccc_{2^g}})$ are indexed by the points in $\mcC=\{0,1\}^g$. The image of $\phi$ is defined as follows
\begin{equation}\label{eq:param}
\begin{matrix} 
u_i = \kappa_{2i-1}-\kappa_{2i},\quad
v_i = \kappa_{2i-1}^2- \kappa_{2i}^2, \quad
w_i = \kappa_{2i-1}^3-\kappa_{2i}^3 \qquad \text{for } i=1,2,\dots, g, \bigskip\\
a_{\mathbf{c}} = \displaystyle\prod_{\substack{i,j\in I\\i<j}} (\kappa_i - \kappa_j) \displaystyle\prod_{\{i\,:\, c_i = 1\}} \lambda_i, \, \quad \text{ where } I = \{2i: c_i = 0\} \cup \{2i-1: c_i = 1\} \quad \text{for } \mathbf{c}\in \mathcal{C}. 
\end{matrix}
\end{equation}
We call the closure of the image of $\phi$ the \emph{main component of $\HC$} and denote it $\HM$.

In the first part of this section, we will explain the geometric intuition that leads us to the definition of the main component $\HM$ by way of the \textit{Abel map} for curves and the \textit{theta divisor}, see \cite{arbarello1985geometry}. In the second part, we will focus on the study of the main component and its connection with the weak and classical Schottky problems.  

\subsection{The Theta Divisor of a Rational Nodal Curve}
Let $X$ be a curve as above and denote by $n_1,n_2,\dots,n_g$ its nodes. The normalization ${\nu:\tilde{X}\to X}$ that separates the $g$ nodes of $X$ is given by a projective line. We can consider $\kappa_1,\kappa_2,\dots,\kappa_{2g}$ to be points on $\PP^1$ and set $\nu^{-1}(n_i):=\{\kappa_{2i-1},\kappa_{2i}\}$. Hence, each rational curve with only nodal singularities corresponds to a copy of $\PP^1$ and $2g$ points on it. This fact motivates that the moduli space of rational nodal curves has dimension $2g-3$, where one subtracts $3$ to account for the dimension of the automorphism group of $\PP^1$. 
A basis of holomorphic differentials for such curves is given by $$\omega_{i}\,\, =\,\, \frac{1}{y} \left(\frac{1}{1-\kappa_{2i}y} - \frac{1}{1-\kappa_{2i-1}y}\right)dy \qquad \text{for } i = 1,2,\dots,g,$$
when fixing $y = 1/x$ as local coordinate. These define a map $\alpha': (\PP^1)^{g-1} \dashrightarrow \CC^{g}$ such that 
$$(y_1, \dots, y_{g-1}) \longmapsto \left(\, \sum_{i=1}^{g-1} \int_0^{y_i} \omega_j\right)_{j=1,2,\dots,g} \quad \text{where}\quad \int_0^{y_i} \omega_j \,=\, \log\left(\frac{1-\kappa_{2j-1}y_i}{1 - \kappa_{2j}y_i}\right).
$$
The (generalized) Jacobian of the curve $X$ is an algebraic torus $(\CC^*)^g$. Exponentiation allows to map in the Jacobian through the map $\CC^g \to (\CC^*)^g$ given by  ($(z_1, \dots, z_g) \mapsto (\exp[z_1], \dots, \exp[z_g])$). The composition gives the Abel map $\alpha:(\PP^1)^{g-1} \dashrightarrow (\CC^*)^g$ defined by
\begin{equation}\label{eq:abelMap}
(y_1, \dots, y_{g-1}) \longmapsto \left( \,\prod_{i=1}^{g-1} \frac{1-\kappa_1y_i}{1-\kappa_2y_i}, \,\,\prod_{i=1}^{g-1} \frac{1-\kappa_3y_i}{1-\kappa_4y_i} ,\dots, \prod_{i=1}^{g-1} \frac{1-\kappa_{2g-1}y_i}{1-\kappa_{2g}y_i}\,\right).
\end{equation}
The theta divisor of $X$ is the image of the Abel map $\alpha$ up to translation.
We will motivate later that we expect each point in the main component of the Hirota variety $\HM$ to correspond (non-injectively) to the choice of a curve in this moduli space and a theta divisor. Following this reasoning, the projection of $\HM$ into the space $(\CC^*)^{2^g}$ has dimension $3g-3$, accounting for the choice of a rational nodal curve and its theta divisor. For each point in this projection, the fiber is a threefold (analogous to the Dubrovin threefold studied in \cite{agostini2021dubrovin}) given by 
\begin{equation*}
\{ \,(\mathbf{u, v, w}) \in \mathbb{W}\PP^{3g-1} \,\,:\,\, \tau(x,y,t)=\theta_{\mcC}(\mathbf{u }x + {\bf v}y + {\bf w}t) \text{ solves } \eqref{eq:hirotabilinearform}\, \}.
\end{equation*}
Thus, the expected dimension of the main component is $2g-3 +g + 3 = 3g$.
This discussion also provides a way to parameterize the main component of the Hirota variety. The idea is that the choice of the curve $X$ yields $2g$ parameters $\kappa_1,\kappa_2,\dots,\kappa_{2g}$, and the family of theta divisors corresponding to shifts of the unknown vector ${\bf z}\in\CC^g$ yields $g$ parameters $\la_1, \dots, \la_g$.
In particular, we are able to compute, using \texttt{Macaulay2}, the theta divisor as a shift (through a change of variables) of the image of the map \eqref{eq:abelMap} up to genus $6$. This is done following the method described in \cite[Corollary 4.8]{michalek2021invitation}. The equation of the theta divisor returned by the code coincides precisely with \eqref{eq:degenerateTheta2} with coefficients ${\bf a} = (a_1,a_2,\dots,a_{2g})$ parameterized in a way which can be shown, through a change of coordinates and some calculations, to be geometrically equivalent to (\ref{eq:param}), up to the $\lambda$ parameters. More precisely, this justifies that the parameterization in \eqref{eq:paramMap} is not unique. The torus action on the theta divisor provides new suitable parameterizations. The approach for computing the theta divisor was inspired by \cite{agostini2021algebraic}. A construction for the theta divisor of a singular curve is also discussed in \cite[Chapter IIIb]{mumford2007tata}.

To understand how the $\lambda$ parameters arise, we consider the Riemann theta function \eqref{eq:degenerateTheta} evaluated at a point ${\bf z} \in \CC^g$ shifted by a vector ${\bf h} = (h_1,h_2,\dots,h_g) \in \mathbb{C}^g$, namely
$$\theta_{\mathcal{C}}(\mathbf{z+h}) = \sum_{\mathbf{c} \in \mcC} a_\mathbf{c} \exp[\mathbf{c}^T\mathbf{h}]\exp[\mathbf{c}^T\mathbf{z}]$$ Both functions $\theta_{\mathcal{C}}(\mathbf{z})$ and $\theta_{\mathcal{C}}(\mathbf{z+h})$ provide solutions to the KP~equation (with nonzero $u_i$), hence they correspond to points in $\HM$. More explicitly, denote $\la_i:=\exp[h_i]$, we see that for a point $(\mathbf{a, (u, v, w)}) \in \HM$, the point $(\mathbf{\tilde{a},(u, v, w)}) \in \HM$, where $\mathbf{\tilde{a}} = (a_{\mathbf{c}}\exp[\mathbf{c}^T\mathbf{h}])= (a_{\mathbf{c}} \prod_{\{i\,:\, c_i = 1\}}\la_i)$. Thus, we conclude that the choice of the theta divisor is exactly represented by the parameterizing variables $\la_1, \dots, \la_g$.

\begin{ex}\label{ex:cube}(g=3)
	This example is intended to clarify the role of the parameters in the case of the $3$-cube. For the $\kappa$ parameters, as in the proof above, we fix $\kappa_1, \dots, \kappa_6 \in \PP^1$. The differentials $\omega_1, \omega_2, \omega_3$, after the coordinate change $y=1/x$, are given by 
	$$
	\omega_1 \,\,=\,\, \frac{1}{y}\left(\frac{1}{1-\kappa_2y}-\frac{1}{1-\kappa_1y} \right)dy,\qquad
	\omega_2 \,\,=\,\, \frac{1}{y}\left(\frac{1}{1-\kappa_4y}-\frac{1}{1-\kappa_3y} \right)dy,\qquad
	$$
	$$  \omega_3 \,\,=\,\, \frac{1}{y}\left(\frac{1}{1-\kappa_6y}-\frac{1}{1-\kappa_5y} \right)dy.$$
	Therefore, the Abel map $\alpha: (\PP^1)^2 \dashrightarrow (\CC^*)^3$ is defined by 
	$$(y_1, y_2) \mapsto \left( \left( \frac{1-\kappa_1y_1}{1-\kappa_2y_1}\right) \cdot \left( \frac{1-\kappa_1y_2}{1-\kappa_2y_2}\right), \left( \frac{1-\kappa_3y_1}{1-\kappa_4y_1}\right) \cdot \left( \frac{1-\kappa_3y_2}{1-\kappa_4y_2}\right), \left( \frac{1-\kappa_5y_1}{1-\kappa_6y_1}\right) \cdot \left( \frac{1-\kappa_5y_2}{1-\kappa_6y_2}\right)\right).$$
	One can find the implicitizing equation cutting out the image of this map in \texttt{Macaulay2} with the~code \begin{verbatim}
	I = ideal(q1*(1-k2*y1)*(1-k2*y2)-(1-k1*y1)*(1-k1*y2),
	          q2*(1-k4*y1)*(1-k4*y2)-(1-k3*y1)*(1-k3*y2),
	          q3*(1-k6*y1)*(1-k6*y2)-(1-k5*y1)*(1-k5*y2));
	J = eliminate(I,{y1,y2})
	\end{verbatim}
	The resulting equation gives exactly the familiar theta function for $g=3$, with the $a_\mathbf{c}$ parameterized by the $\kappa_i$'s.
	For the $\la$ parameters,
	we consider the theta functions
	\begin{align}\label{eq:cubetheta}\theta_{\mathcal{C}}(\mathbf{z}) = a_{000} &+ a_{100}\exp[z_1] + a_{010}\exp[z_2] + a_{001}\exp[z_3] + a_{110}\exp[z_1 + z_2]\\ & + a_{101}\exp[z_1 + z_3]+ a_{011}\exp[z_2 + z_3] + a_{111}\exp[z_1 + z_2 + z_3]\nonumber\end{align} and \begin{align*}\theta_{\mathcal{C}}({\bf z + h}) = a_{000} &+ a_{100}\exp[h_1]\exp[z_1] + a_{010}\exp[h_2]\exp[z_2] + a_{001}\exp[h_3]\exp[z_3]\\ &+ a_{110}\exp[h_1 + h_2]\exp[z_1 + z_2]+ a_{101}\exp[h_1 + h_3]\exp[z_1 + z_3] \\&+ a_{011}\exp[h_2+h_3]\exp[z_2 + z_3] + a_{111}\exp[h_1 + h_2 + h_3]\exp[z_1 + z_2 + z_3].\end{align*}
	Letting $ \lambda_i:=\exp[h_i]$, we have 
	\begin{align*}
	\theta_{\mathcal{C}}({\bf z + h}) = a_{000} &+ \la_1a_{100}\exp[z_1] + \la_2a_{010}\exp[z_2] + \la_3a_{001}\exp[z_3] + \la_1\la_2a_{110}\exp[z_1 + z_2] \\ &+ \la_1\la_3a_{101}\exp[z_1 + z_3]+ \la_2\la_3a_{011}\exp[z_2 + z_3] + \la_1\la_2\la_3a_{111}\exp[z_1 + z_2 + z_3].\end{align*}
	This gives us $g$ parameterizing factors $\la_i$ with $i = 1,2,\dots, g$ for the variables $a_{\bf c}$. 
\end{ex}
The code used in Example \ref{ex:cube} is available at the repository website \texttt{MathRepo} \cite{fevola2022mathematical} of MPI-MiS via the link 
\begin{equation}\label{eq:mathrepo}
\texttt{\href{https://mathrepo.mis.mpg.de/HirotaVarietyRationalNodalCurve}{https://mathrepo.mis.mpg.de/HirotaVarietyRationalNodalCurve}}.
\end{equation}

\subsection{Geometry of the Main Component}
The following result justifies the definition of the main component $\HM$ and provides a connection with soliton solutions to the KP equation.

\begin{thm}\label{thm:imagecontained}
	Consider the map $\phi$ given in (\ref{eq:paramMap}). This is a birational map onto its image, which is an irreducible subvariety of $\HM$ and has dimension $3g$. 
\end{thm}
\begin{proof}
	Let $I$ be as in (\ref{eq:param}). Let $K \subseteq \CC^{3g}$ be the closed set where at least two of the $\kappa_i$ coincide. The expression of $\tau(x,y,t)=\theta_\mathcal{C}( {\bf u} x + {\bf v} y + {\bf w} t )$ where $({\bf a},({\bf u},{\bf v}, {\bf w}))$ is attained as the image of a point in $\CC^{3g}\setminus K$ through the map $\phi$, described in \eqref{eq:paramMap}, is a point in the Hirota variety $\mcH_\mcC$ since 
	it can be expressed as a $(g,2g)$-soliton \cite{kodama2004young} for the matrix
	\begin{equation}\label{eq:matrixSoliton}
	A \,=\, \begin{pmatrix}
	\lambda_1 & 1 & 0 & 0 & 0 & 0 & \dots & 0 & 0 \\
	0 & 0 & \lambda_2 & 1 & 0 & 0 & \dots & 0 & 0\\
	0 & 0 & 0 & 0 & \lambda_3 & 1 & \dots & 0 & 0\\
	\vdots & \vdots & \vdots & \vdots & \vdots & \vdots & \ddots & \vdots & \vdots \\
	0 & 0 & 0 & 0 & 0 & 0 & \dots & \lambda_g & 1\\
	\end{pmatrix}.
	\end{equation}
	
	Indeed, if we denote $E_i:=\exp (\kappa_i x + \kappa_i^2 y + \kappa_i^3 t)$, plugging the parameterization in $\theta_{\mathcal{C}}( {\bf u} x + {\bf v} y + {\bf w} t)$ we obtain:
	
	\begin{equation}\label{eq:g2gsoliton}
	\tilde{\tau}_{\mathcal{C}}(x,y,t) \,=\, \frac{1}{E_2E_4\cdots E_{2g}}\biggl(\displaystyle\sum_{I} \displaystyle\prod_{i< j \in I} (\kappa_i - \kappa_j) \displaystyle\prod_{\{i\,:\, c_i = 1\}} \lambda_i \prod_{i\in I} E_i \biggr),
	\end{equation}
where the sets $I$ defining the sum are the same ones defined in (\ref{eq:param}). The extraneous exponential factor $(E_2E_4\cdots E_{2g})^{-1}$ disappears after
	we pass from $\tilde \tau(x,y,t)$ to $\partial_x^2 \, {\rm log}(\tilde \tau(x,y,t))$. Both
	versions of the $(g,2g)$-soliton satisfy the Hirota's bilinear form and they represent the same solution to the KP equation. Hence, it follows that the image of $\CC^{3g}$ through the map $\phi$ is contained in $\HC$.
	
	Furthermore, the map $\phi$ is invertible outside the closed set where the $u_i$'s vanish: given a point $(\mathbf{a, (u, v, w)})$ in the image, one can write 
	\begin{equation}\label{eq:invertedKappas}
	\kappa_{2i-1} \,=\, \frac{u_i^2 + v_i}{2u_i} \qquad \text{and} \qquad \kappa_{2i} \,=\, \frac{v_i - u_i^2}{2u_i},
	\end{equation}
	and the $\lambda_i$'s can be obtained sequentially, starting from $\lambda_1$, by plugging in the values for the $\kappa$ variables into the $a_i$'s. Hence, we can conclude that the map $\phi$ is birational. This implies that the closure of the image is irreducible and has dimension $3g$.
\end{proof}
Notice that the method above gives a way to parameterize solutions arising from a genus $g$ rational nodal curve as $(g,2g)$-solitons. This is also consistent with \cite[Example 29]{AFMS}.

In what follows, we show that $\HM$ is an irreducible component of $\HC$ whose points correspond to genus $g$ rational nodal curves. This is equivalent to solving a version of the \emph{weak Schottky problem}. In fact, $\HM$ parameterizes some solutions to the KP equation arising from irreducible rational nodal curves of genus $g$, and hence corresponds to a variety containing the locus of Jacobians of such curves as an irreducible component.  

\begin{thm}\label{thm:irredcomp}
	For genus $g\leq 9$, the subvariety $\HM$ is an irreducible component of the Hirota variety. 
\end{thm}
\begin{proof}
	The proof is mainly computational. A direct computation performed in \texttt{Macaulay2} shows that the Jacobian matrix of the Hirota variety $\HC$ evaluated at the image of  a general point in $\CC^{3g}$ through $\phi$ has rank $r-(3g)$, where $r$ is the dimension of the space $(\CC^*)^{2^g}\times\mathbb{W}\PP^{3g-1}$. Therefore, the map $\phi$ is dominant into the main component $\HM$.
\end{proof}
The code used for the proof above can be found at \eqref{eq:mathrepo}.
The rapid increase in the number of variables ${\bf a}= (a_{\ccc_1},a_{\ccc_2},\dots,a_{\ccc_{2^g}})$ makes the computation difficult for higher genera. One possible way to compute the ideal defining the variety $\HC$ is to use the condition provided by the Hirota's bilinear form \eqref{eq:hirotabilinearform}. However, such a computation becomes expensive when the genus is larger than $7$. To avoid this computation, we implement the equations cutting out $\HC$ via the combinatorial description provided in \cite[Section 3]{AFMS}. 

The parameterization for the main component of the Hirota variety associated to the $3$-dimensional cube was already developed in \cite[Example 8]{AFMS}, although the definition of the main component was not clearly stated yet. We revisit this example, working out the details providing the intuition behind Theorem \ref{thm:imagecontained}. 
\begin{ex}[$g=3$]
	The $3$-cube associated to the dual graph of a rational nodal quartic is the support of the degenerate theta function \eqref{eq:cubetheta}.
	The Hirota variety $\HC$ in $(\CC^*)^8 \times \mathbb{WP}^8$ is cut out by the $19$ polynomial equations displayed in \cite[Example 8]{AFMS}. A direct computation shows that the ideal defining  $\HC$ has five associated primes, hence the Hirota variety has five irreducible
	components in $(\CC^*)^8 \times \mathbb{W} \PP^8$.
	If we restrict to the main component, we can observe that the additional quartic relation
	\begin{equation}\label{eq:THEquartic}
	a_{000} a_{110} a_{101} a_{011} \,=\,a_{100} a_{010} a_{001} a_{111} 
	\end{equation}
	holds. 
	The main component has dimension $9$, while its image in $\mathbb{WP}^8$ has dimension $5$ and is defined by the equations 
	$u_i^4+3v_i^2-4u_iw_i$ for $i=1,2,3$, with fibers given by cones over $\PP^1 {\times} \PP^1 {\times} \PP^1$.
	The following parametric representation of the main component in $\HC$
	$$  \begin{matrix} u_1 = \kappa_1-\kappa_2\,,\quad
	v_1 = \kappa_1^2- \kappa_2^2 \, , \quad
	w_1 = \kappa_1^3-\kappa_2^3 \, , \\
	u_2 = \kappa_3-\kappa_4 \, ,\quad
	v_2 = \kappa_3^2-\kappa_4^2 \, , \quad
	w_2 = \kappa_3^3 - \kappa_4^3\, , \\
	u_3 = \kappa_5- \kappa_6 \, , \quad
	v_3 = \kappa_5^2-\kappa_6^2\, ,\quad
	w_3 = \kappa_5^3-\kappa_6^3 \, , \\
	a_{111} = (\kappa_3-\kappa_5)(\kappa_1-\kappa_5) (\kappa_1-\kappa_3)  \lambda_1 \lambda_2 \lambda_3 \,, \,\,\,\,
	a_{011} = (\kappa_3-\kappa_5)(\kappa_2-\kappa_5)(\kappa_2-\kappa_3)  \lambda_2 \lambda_3\,, \\
	a_{101} = (\kappa_4-\kappa_5)(\kappa_1-\kappa_5)(\kappa_1-\kappa_4)  \lambda_1 \lambda_3\, ,\,\,\,\,
	a_{001} = (\kappa_4-\kappa_5)(\kappa_2-\kappa_5)(\kappa_2-\kappa_4)  \lambda_3\,, \\
	a_{110} = (\kappa_3-\kappa_6)(\kappa_1-\kappa_6)(\kappa_1-\kappa_3)  \lambda_1 \lambda_2 \, , \,\,\,\,
	a_{010} = (\kappa_3-\kappa_6)(\kappa_2-\kappa_6)(\kappa_2-\kappa_3)  \lambda_2 \, ,  \\
	a_{100} = (\kappa_4-\kappa_6)(\kappa_1-\kappa_6)(\kappa_1-\kappa_4)  \lambda_1 \, , \,\,\,\,
	a_{000} = (\kappa_4-\kappa_6)(\kappa_2-\kappa_6)(\kappa_2-\kappa_4)  
	\end{matrix}
	$$
	provides a way of identifying the $\tau$-fuction arising from \eqref{eq:cubetheta} with a $(3,6)$-soliton for the matrix
	\begin{equation}
	\label{eq:dreisechs}
	A \,\, = \,\, \begin{small} \begin{pmatrix} 
	\lambda_1 & 1 & 0  & 0 & 0 & 0 \\
	0 & 0 & \lambda_2 & 1 & 0  & 0 \\
	0 & 0 & 0 & 0 & \lambda_3 & 1  \end{pmatrix} . \end{small} \end{equation}

\end{ex}
The description given in Example \ref{ex:cube} of the projection of the main component $\HM$ into the space $\mathbb{\mathbb{W}\PP}^{3g-1}$ reveals itself to be true for any genus:
\begin{prop}
	The projection of $\HM$ into $\mathbb{\mathbb{W}\PP}^{3g-1}$ is a $(2g-1)$-dimensional variety defined by the vanishing of $ u_i^4 + 3v_i^2 - 4u_iw_i$ for $i = 1,2,\dots,g$.
\end{prop}
\begin{proof}
	One direction (i.e., that the relations $u_i^4 + 3v_i^2 - 4u_iw_i$ hold in the projection) is immediate, as these are polynomials defining the Hirota variety (see Lemma \ref{lem:gquadrics}) that do not include any $a_i$. For the other direction, it suffices to exhibit a point in $\HM$ for any $(\mathbf{u, v, w})$ which satisfy  $u_i^4 + 3v_i^2 - 4u_iw_i$ for all $i\in [g]$. Using the inverse map given in \eqref{eq:invertedKappas}, given any $u_i, v_i, w_i$ satisfying $u_i^4 + 3v_i^2 - 4u_iw_i\,=\,0$, we uniquely determine (up to a scaling factor) $\kappa_{2i-1}, \kappa_{2i}$. We can then choose arbitrary $\lambda_1, \dots, \lambda_g$ to get a point $(\mathbf{a,( u, v, w)})$, so we are done.
\end{proof}

We conclude this section by stating two conjectures that consolidate and generalize the results above to any genus. Indeed, a generalization of  Theorem \ref{thm:irredcomp} to any genus would provide a solution to the weak Schottky problem for rational nodal curves of genus $g$. The analogous results in the case of smooth curves were proven by Dubrovin in \cite{dubrovin1982the,dubrovin1981theta}.

\begin{conj}[Weak Schottky Problem]
\label{conj:dim}
	For any genus $g$, the main component of the Hirota variety $\HM$ is a $3g$-dimensional irreducible component of $\HC$ with a parametric representation given by (\ref{eq:param}).
\end{conj}

\begin{conj}[Strong Schottky Problem]\label{conj:imageComp}
	$\HM = \HC^I$.
\end{conj}

Notice that one direction of the Schottky problem is immediate from the proof of Theorem \ref{thm:imagecontained}: a sufficiently generic choice of $\kappa_i$ ensures that, in the image, the $u_i$ are nonzero, thus $\HM \subseteq \HC^I$. The other direction is more difficult. To prove it, one would need to show that any point in $\HC^I$ can be parameterized as in (\ref{eq:param}).

\section{Combinatorics of the Hirota Variety}\label{sec:combinatorics}
The results in this section describe in detail facts that we use in many of the proofs in Section \ref{sec:maincomp}. We begin by explaining how the nodal singularities on the curve $X$ induce the degenerate theta function in \eqref{eq:degenerateTheta}. To each curve of genus $g$ we associate a metric graph $\Gamma$ of genus $g$. This graph has a vertex for each
irreducible component, one edge for each intersection point between two components, and a node on an irreducible component gives a loop on the corresponding vertex. Hence, if $X$ is a rational nodal curve of genus $g$, the corresponding  metric graph $\Gamma$ is given by one unique node and $g$ cycles. Figure \ref{metric} illustrates an example when $g=5$. 
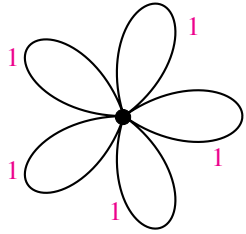
\begin{figure}[htbp]
	\centering
	
	\begin{tikzpicture}[scale=0.5]
	\begin{polaraxis}[grid=none, axis lines=none]
	\addplot[style={ultra thick},mark=none,domain=0:360,samples=300] { abs(cos(5*x/2))};
	\end{polaraxis}
	\draw[fill] (3.4,3.4) circle [radius=0.2];
	\node[color=magenta] at (3.2,0.9) {1};
	\node[color=magenta] at (5.25,5.8) {1};
	\node[color=magenta] at (.5,2) {1};
	\node[color=magenta] at (.5,5) {1};
	\node[color=magenta] at (5.9,2.33) {1};
	\end{tikzpicture}
	\caption{The metric graph for an irreducible rational nodal curve of genus 5.}
	\label{metric}
\end{figure}

The way we read off the tropical Riemann matrix $Q$ from a tropical curve is described in \cite{bolognese2017curves,chan2012combinatorics}. 
More explicitly, to determine such positive definite real symmetric $g \times g$ matrix one needs to fix a basis of cycles in $\Gamma$ and write these
as the $g$ rows of a matrix $\Lambda$ whose columns are
labeled by the edges of the graph. Let $\Delta$ be the diagonal matrix
whose entries are the edge lengths of $\Gamma$, then we define
$Q := \Lambda \Delta \Lambda^T$. 
Hence, for $X$ being a rational nodal curve the matrices $\Lambda$ and $\Delta$ both equal the identity matrix $I_g$, returning $Q = I_g$ as the tropical Riemann matrix.
Theorem 3 in \cite{AFMS} describes the Riemann theta function \eqref{eq:thetaFunction} when the curve degenerates. In our case, the distance induced on $\RR^g$ is the Euclidean distance and we can fix the point ${\bf a} = (\frac{1}{2},\frac{1}{2},\dots,\frac{1}{2}) \in \RR^g$ as a vertex of the \textit{Voronoi cell} for $I_g$ given by the cube with vertices $(\pm\frac{1}{2},\pm\frac{1}{2},\dots,\pm\frac{1}{2})$. Under these hypotheses, the support of the degenerate theta function is the Delaunay set
\begin{equation}\label{eq:delaunay}
\mathcal{C} \,\,=\,\, \mathcal{D}_{{\bf a},I_g} \,\,=\,\, \{\,\ccc \in \ZZ^g \;: \; \norm{\bf a}^2 = \norm{{\bf a}-\ccc}^2\,\} \,\,=\,\, \{0,1\}^g,
\end{equation}
where $\norm{\cdot}$ denotes the Euclidean norm. The degenerate theta function is then a finite sum of $2^g$ exponentials 
\begin{equation}\label{eq:degenerateTheta2}
\theta_{\mathcal{C}}({\bf z}) \,\, = \, \sum_{\ccc \in \{0,1\}^g} a_{\ccc} \exp \big[ \ccc^T {\bf z}\big], \quad \hbox{where}\quad  a_{\ccc} \,=\, \exp \biggl[\frac{1}{2} \ccc^T R_0 \ccc\biggr].
\end{equation}
Here, the matrix $R_0$ is the limit of a matrix $R_{\epsilon}$ which is a symmetric $g\times g$ matrix with entries given by complex analytic functions in $\epsilon$ converging for $\epsilon\to 0$. This comes from degenerating the family of Riemann matrices given by $B_{\epsilon} = -\frac{1}{\epsilon}Q + R_\epsilon$. The matrices $B_\epsilon$ lie in the Schottky locus. Corollary 6 in \cite{AFMS} describes the polynomials defining the Hirota variety associated to $\mathcal{C}$. These correspond to points in the set 
\begin{equation*}
\mathcal{C}^{[2]}\,\, = \,\,
\bigl\{ \,{\bf c}_k + {\bf c}_\ell \,\,: \,\, 1 \leq k < \ell \leq  m \,\bigr\}
\,\, \subset \,\, \ZZ^g,
\end{equation*}
where one says that a point ${\bf d}$ in $\mathcal{C}^{[2]}$ is {\em uniquely attained} if there exists precisely one index pair $(k,\ell)$ such~that ${\bf c}_k + {\bf c}_\ell = {\bf d}$. Let $P(x, y, t) = x^4 + 3y^2 - 4xt$. The polynomials defining $\HC$ are explicitly given by the quartics 
\begin{equation}\label{eq:quartic}
P_{k \ell}({\bf u},{\bf v},{\bf w}) \,\,\, := \,\,\,
P \bigl( \,({\bf c}_k - {\bf c}_\ell) \cdot {\bf u},
\,({\bf c}_k - {\bf c}_\ell) \cdot {\bf v},
\,({\bf c}_k - {\bf c}_\ell) \cdot {\bf w} \bigr),
\end{equation}
when ${\bf d} = {\bf c}_k + {\bf c}_\ell$ is uniquely attained, and by 
$$\sum_{1 \leq k < \ell \leq m \atop  \mathbf{c}_k + \mathbf{c}_\ell = \mathbf{d}} P_{k \ell}({\bf u},{\bf v},{\bf w}) a_ka_\ell,$$ 
when ${\bf d}\in \mathcal{C}^{[2]}$ is not uniquely attained. Hence, to better understand the variety $\HC$ we investigate the elements in the set $\mathcal{C}^{[2]}$. We say that a point $\ccc \in \mathcal{C}^{[2]}$ is attained $n$ times if there exist $n$ distinct pairs $(k,\ell)$ such~that ${\bf c}_k + {\bf c}_\ell = {\bf d}$. 

\begin{prop}\label{countattained}
	A point $\mathbf{c}=(c_1, \dots, c_g)$ in $\mathcal{C}^{[2]}$ is attained $2^{d-1}$ times, where ${d =|\{i: c_i = 1\}|}$.
\end{prop}
\begin{proof}
	For a point $\mathbf{c} \in \mathcal{C}^{[2]}$, consider the set of indices $I = \{i: c_i \neq 1\}$. Suppose now $\mathbf{c} = \mathbf{c}_1+\mathbf{c}_2$ for some points $\mathbf{c}_1, {\bf c}_2 \in \mathcal{C}$. Then, for any $i \in I$, if $c_i = 0$ then $c_{1i} = c_{2i} = 0$, while if $c_i = 2$ then $c_{1i} = c_{2i} = 1$. In the first case, this means that both $\mathbf{c}_1$ and $\mathbf{c}_2$ lie on the face of the $g$-cube defined by the $i$-th coordinate hyperplane $x_i = 0$. In the latter case, $\mathbf{c}_1,{\bf c}_2$ lie on the face defined by $x_i = 1$. The full set  $I$ of indices corresponding to elements $\neq 1$ defines a set of restrictions on $x_i$ for $i \in I$. In fact, it defines a face of codimension $|I|$ (and thus dimension $d$) that $\ccc_1$ and $\ccc_2$ lie on.
	Let $[g] = \{1,\dots,g\}$. For the indices $i \in [g]\setminus I$, we have exactly one of $c_{1i}, c_{2i}$ equal to $1$. This gives exactly $2^{d-1}$ such pairs. These pairs can be viewed as diagonals of the faces defined by the restrictions given by $I$: they are the points which are distinct from one another in each coordinate except for the ones fixed by the face.   
\end{proof}

As discussed in the proof, the points in $\mathcal{C}^{[2]}$ correspond to $d$-dimensional faces of the $g$-cube, where $d$ is the number of coordinates equal to $1$. Hence $|\mathcal{C}^{[2]}| = \sum_{d=1}^g 2^{g-d}\binom{g}{d}$ and the pairs that sum to points in $\mathcal{C}^{[2]}$ correspond to diagonals of the associated face. Another way to count $|\mathcal{C}^{[2]}|$ is to observe that it consists exactly of the points in $\{0, 1, 2\}^g$ which have at least one $1$, so there are $3^g - 2^g$ of them.

We now investigate the polynomials arising from the points in $\mathcal{C}^{[2]}$. The points which are attained once correspond to the edges (one-dimensional faces) of the cube and the unique pair that adds up to such a point are the two vertices $\mathbf{c}_k, \mathbf{c}_\ell$ comprising the edge. Hence, these points contribute the quartic \eqref{eq:quartic}, one can notice that this quartic depends only on the difference $\mathbf{c}_k - \mathbf{c}_\ell$, which is the same for all edges going in the same direction (that is, all edges whose corresponding point in $\mathcal{C}^{[2]}$ has the unique $1$ at the same index). This reasoning yields the immediate result 

\begin{lemma}\label{lem:gquadrics}
The set $\mathcal{C}^{[2]}$ contains $g2^{g-1}$ points which are uniquely attained. These contribute as generators of the ideal defining the Hirota variety $\mathcal{H}_{\mathcal{C}}$ (and therefore, also $\mathcal{H}_{\mathcal{C}}^I$) with $g$ quartics of the form $u_i^4 - 4 u_i w_i + 3 v_i^2$, for $i = 1,2,\dots , g$.
\end{lemma}

Recall that the Hirota variety lies in the ambient space $(\CC^*)^{2^g} \times \mathbb{WP}^{3g-1}$. The coordinate ring is $\CC[\mathbf{a}^{\pm 1}, \mathbf{u, v, w}]$  and the ideal defining $\mathcal{H}_\mathcal{C}$ has $g + \sum_{d = 2}^g 2^{g-d}\binom{g}{d}$ generators, one for each edge direction, and one for each face of every dimension from 2 up to $g$.

The combinatorics of the cube has already been shown to be important when studying the generators of the Hirota variety. In what follows, we discuss more of the combinatorics of the cube as it relates to the main component, presenting a more general version of Lemma \ref{lem:gquadrics}. We begin with some definitions. 

We denote $\hat{\mathcal{C}}$ the convex hull of the set $\mathcal{C} = \{0,1\}^g$ in $\RR^g$ with coordinates $x_1,x_2,\dots,x_g$. A $d$-dimensional face of $\hat{\mathcal{C}}$ is determined by fixing $g-d$ indices of the vertices defining it. These indices are precisely the ones appearing in the coordinate hyperplanes that define each face.
 We call the \emph{direction} of the face the set of indices $I = \{i_1,i_2,\dots,i_d\}$ that are not fixed. Furthermore, if two $d$-dimensional faces have the same direction, we define the \emph{difference} between them to be the set $J$ of fixed indices in the two faces which are different.

\begin{ex}[$g=3, d=1$]
Let $\mathcal{C} = \{0,1\}^3$. The direction of an edge is given by a set with one element, namely the index of the standard basis vector to which the edge is parallel. For instance, the edges $\text{conv}((0,1,0),(0,1,1))$ and $\text{conv}((1,1,0),(1,1,1))$ are determined respectively by the hyperplanes $\{x_1=0,\,\,x_2=1\}$, and $\{x_1=x_2=1\}$. Hence, they have the same direction given by $I= \{3\}$. The difference of two edges on the same two dimensional face is also a one-element set, consisting of the index of the second standard basis vector defining the face (in addition to the standard basis vector given by the direction). Thus, the two edges above have difference $J = \{1\}$ corresponding to the coordinate hyperplane $x_1$ that determines different entries for the first coordinate of the vertices spanning the edges.
\end{ex}

In the following result, we restrict to the main component $\HM$ of the Hirota variety. We are interested in the points in $\HM$ that also verify the quartic relations
\begin{equation}\label{eq:a-relations}
a_{\ccc_1}a_{\ccc_2}a_{\ccc_3}a_{\ccc_4} \, = \,a_{{\bf d}_1} a_{{\bf d}_2} a_{{\bf d}_3}a_{{\bf d}_4}, \quad \text{with } \sum_{i=1}^4\ccc_i=\sum_{i=1}^4 {\bf d}_i,\;\; \sum_{i=1}^4\ccc_i^2=\sum_{i=1}^4 {\bf d}_i^2,
\end{equation}
where $\ccc_i,{\bf d}_i$ are points in $\{0,1\}^g$. For $g=3$, there exists a unique relation of this type, namely the one in \eqref{eq:THEquartic}. 
In particular, when the $a_i$ are exponentials of the form $a_\mathbf{c} = \exp[\frac{1}{2} \mathbf{c}R_0\mathbf{c}^T]$, as in \eqref{eq:degenerateTheta2}, then they verify these quartic relations. More generally, one has

\begin{lemma}
	The closure of the image of the map $\psi : \text{Sym}^2(\CC^g) \to \PP^{2^g-1}$ defined by 
	\begin{equation}\label{eq:embeddingR}
	R \,\,\mapsto\,\, \left(\,a_{\ccc}\,=\,\exp\left[\frac{1}{2}\ccc^TR\ccc \right]\,\right)_{\ccc \in \{0,1\}^g}    
	\end{equation}
	is cut out by the equations in \eqref{eq:a-relations} and the additional equation $a_{{\bf 0}}=1$.
\end{lemma}

\begin{proof}
An immediate computation shows that the points in the image of $\psi$ verify the relations in \eqref{eq:a-relations} and $a_{{\bf 0}}=1$. To show the converse, we consider a point ${\bf a} = [a_{{\bf c}_1}:a_{{\bf c}_2}:\dots:a_{{\bf c}_{2^g}}]\in \PP^{2^g-1}$, with entries indexed by points in $\mathcal{C}$, such that it verifies the desired equations. Define the matrix $R \in \text{Sym}^2(\CC^g)$ with entries given by 
$$R_{ii}\,=\,2\log a_{e_i}\quad \text{and} \quad R_{ij}\,=\,\log\frac{ a_{e_i+e_j}}{a_{e_i}a_{e_j}} \quad \text{for } i,j \in [g], \,\, i\neq j,$$
where $e_i$ denotes the $i$-th vector in the standard basis of $\ZZ^g$ and $\log$ denotes the natural logarithm. In this way, by definition we have 
$$ a_{e_i} \,=\, \exp\left[ \frac{1}{2} e_i^T R e_i\right]\quad \text{and} \quad a_{e_i+e_j} \,=\, \exp\left[ \frac{1}{2} (e_i+e_j)^T R (e_i+e_j)\right] \quad \text{for } i,j \in [g],\,\, i\neq j.$$ 
Notice that the points in $\mathcal{C}$ are indexed by all possible subsets of the set $[g]=\{1,2,\dots,g\}$, and they are all of the form ${\bf c}_I = \sum_{\{i\in I\}} e_i$, with ${\bf c}_{\emptyset} = {\bf 0}$. Hence, we proceed by induction on the size $n$ of the support $I$. The cases $n=1,2$ have been verified above. Therefore, we assume that the $a_{c_I}$ have the desired form for any $I\subset [g]$ with $|I|\leq n-1$.

It is sufficient to prove the claim for the element $a_{e_1+\dots+e_n}$. By hypothesis, it verifies a relation of the form
\begin{equation}\label{eq:quartic_proof}
a_{e_1+\dots+e_n}\, = \, \frac{a_{e_1+\dots+ e_{n-1}} \cdot a_{{e_1}+\dots+ e_{n-2} + e_n} \cdot a_{e_{n-1}+e_n} \cdot a_{{\bf 0}} }{a_{e_1+\dots+e_{n-2}}\cdot a_{e_{n-1}} \cdot a_{e_n}},
\end{equation}
where all the factors on the right-hand side of the equality are of the prescribed exponential form. Observe that, for ${\bf c}_I$ as above, in general one has
$$	a_{{\bf c}_I} \, = \, \exp \left[ \frac{1}{2} \biggl( \sum_{i\in I} e_i \biggr)^T R \biggl(\sum_{i\in I} e_i\biggr) \right]\, = \,
	\exp \biggl[ \frac{1}{2} \biggl(\sum_{i\in I} R_{ii} + \sum_{\substack{i,j\in I\\i\neq j}} R_{ij}\biggr) \biggr].$$
Substituting such exponentials in the right-hand side of (\ref{eq:quartic_proof}), we obtain that 
$$a_{e_1+\dots+e_n}\, = \, \exp \left[ \frac{1}{2} (e_1+\dots + e_n)^T R (e_1+\dots + e_n) \right].$$
This concludes the proof.
\end{proof}

\begin{thm}\label{thm:directions}
	There are $\binom{g}{d}$ face directions for each dimension $d$, and all faces with the same direction contribute the same quartic, up to a multiple, to the ideal defining $\HM$. 
\end{thm}
\begin{proof}
Consider two $d$-dimensional faces of the $g$-cube with the same direction. This means that their corresponding points $\mathbf{c}_1, \mathbf{c}_2 \in \mathcal{C}^{[2]}$ have $1$s in exactly the same positions. Both points have $2^{d-1}$ pairs which sum to them, and these pairs can be put in a correspondence. Namely, for a pair $\ccc_{k}, \ccc_{\ell}$ that sums to $\ccc_1$, the pair $\tilde{\ccc}_{k}, \tilde{\ccc}_\ell$ is a pair that sums to $\ccc_2$, where $\tilde{\mathbf{a}}$ is obtained from $\mathbf{a}$ by changing the entry from $0$ to $1$ (or vice-versa) for every index in the difference of the two faces. For the two pairs, $(\ccc_k, \ccc_\ell)$ and $(\tilde{\ccc}_k, \tilde{\ccc}_\ell)$, their corresponding quartic $P_{kl}$ is the same, since it is easy to see that $\ccc_k - \ccc_\ell = \tilde{\ccc}_k - \tilde{\ccc}_\ell$.

Recall from Section \ref{sec:combinatorics}, that the generators of the ideal $\mathcal{I}(\mathcal{H}_{\mathcal{C}})$ corresponding to $d$-dimensional faces with $d>1$ are of the form  
\begin{equation}
\label{eq:nonunique}
\sum_{\genfrac{}{}{0pt}{}{1 \leq k < \ell \leq m}{
		{\bf c}_k + {\bf c}_\ell \,=\, {\bf d}}} \!\! P_{k\ell} ({\bf u},{\bf v},{\bf w}) \, a_k a_\ell,
\end{equation}
where ${\bf d} \in \mathcal{C}^{[2]}$ is not uniquely attained. In what follows, we show that, when we restrict to the main component $\HM$, a $d$-dimensional face with direction $D$ contributes the same quartic, up to a multiple, as the face corresponding to the point $\ccc_D = \sum_{\{i \in D\}} e_i$. We have already shown that the $P_{k\ell}$ are the same for faces with the same direction. Thus, it is sufficient to show that the polynomial $$a_{\ccc_k}a_{\ccc_\ell} - d \cdot a_{\tilde{\ccc}_k}a_{\tilde{\ccc}_\ell}$$ 
is in the ideal defining $\HM$,  where $d$ is a factor (in fact, a product of some $a_{\ccc}$'s) which does not depend on $k, \ell$, instead it depends only on the difference and direction of the two faces. 

For ease of computations, we fix a direction $D$ and we will take one of the faces ($F_1$) to be the face corresponding to the point $\ccc_1 = \sum_{\{i \in D\}} e_i$. The other face ($F_2$) is a face with direction $D$ and difference $E$ from $F_1$. Hence, the point corresponding to $F_2$ is $\ccc_2 = \sum_{\{i \in D\}} e_i + \sum_{\{i \in E\}} 2e_i$. Since we will show that the quartic contributed by $F_2$ is the same as the one contributed by $F_1$ up to a multiple, this will show that all faces with the same direction contribute essentially the same polynomial to the ideal defining $\HM$.

Recall that the $a_{\ccc}$ are given by $\exp[\frac{1}{2} \ccc R\ccc^T]$, where we write $R$ for $R_0$ from (\ref{eq:degenerateTheta2}). Consider a pair $\ccc_k, \ccc_\ell \in \mathcal{C}$ such that $\ccc_1 = \ccc_k + \ccc_\ell$. Then, there exist two disjoint subsets $D_1,D_2\subseteq D$ such that $D = D_1 \cup D_2$, and $\ccc_k = \sum_{\{i \in D_1\}} { e}_i$ and $\ccc_\ell = \sum_{\{i \in D_2\}} {e}_i$. It follows that $\tilde{\ccc}_k = \sum_{\{i \in D_1 \cup E\}} {e}_i$ and $\tilde{\ccc}_\ell = \sum_{\{i \in D_2 \cup E\}} {e}_i$.

We will now use the linear algebra fact that for a symmetric $g\times g$ symmetric matrix, the following hold 
$$\left(\,\sum_{i \in I} {e}_i\right) R\, \biggl(\,\sum_{j \in J}{e}_j\biggr)^{T} = \sum_{i \in I, j \in J}R_{ij}$$
Therefore, we have 
\begin{equation*}
a_{\tilde{\ccc}_k}a_{\tilde{\ccc}_\ell} \,\,=\,\, \exp\biggl[\, \frac{1}{2}\sum_{i, j \in D_1 \cup E} R_{ij}\biggr]\exp\biggl[\,\frac{1}{2}\sum_{i, j \in D_2 \cup E} R_{ij}\biggr]
\,\, =\,\, \exp\biggl[\,\frac{1}{2}\sum_{i, j \in D \cup E} R_{ij}\biggr]\exp\biggl[\,\frac{1}{2}\sum_{i, j \in E} R_{ij}\biggr],
\end{equation*}
 which one can easily see is a multiple of $\exp[\,\frac{1}{2}\sum_{\{i, j \in D\}} R_{ij}] = a_{{\ccc}_k}a_{{\ccc}_\ell}$ which only depends on the sets $D$ and $E$, as desired.
\end{proof}

One can observe that Theorem \ref{thm:directions} is a generalization of Lemma \ref{lem:gquadrics} to equations arising from points in $\mathcal{C}^{[2]}$ which correspond to higher dimensional faces. This holds for points 
in the main component $\HM$. Moreover, Theorem \ref{thm:directions} reduces the number of potentially non-redundant relations holding in the ideal defining the variety $\HM$ to $2^g-1$. This is also the codimension of $\HM$ inside its ambient space $\KWP$. Our code in \eqref{eq:mathrepo} verifies that this set of generators defines a variety which is a complete intersection, of which $\HM$ is a component for $g\leq 9$.

\subsection{Quartic Relations and The Schottky Locus}
We conclude with a discussion relating the quartic relations among the $a_i$ parameters and the Schottky problem. The \textit{classical Schottky problem}, studied by Riemann and Schottky, requires one to write down the defining equations for the \textit{Schottky locus} $\Ig$, namely the subset of abelian varieties corresponding to Jacobians of curves. In fact, the \textit{second order theta constants} \cite{vangeemen1998constants} define an embedding of $\mathcal{A}_g$ into a projective space 
$\mu: \mathcal{A}_g \hookrightarrow \PP^{2^g-1}$ and to solve the Schottky problem one aims to determine the defining ideal of $\mu(\Ig)\subset \mu(\mathcal{A}_g)$. Hence, the Schottky problem concerns the maps 
\begin{equation}\label{eq:mapsSchottky}
\mcM_g \,\,\xrightarrow{J}\,\ \mcA_g\,\, \hookrightarrow \,\,\PP^{2^g-1}.
\end{equation}
Moreover, the space $\mcA_g$ is parameterized by the \textit{Siegel upper-half space} $\mathfrak{H}_g$, namely the set of complex symmetric $g \times g$ matrices with positive definite imaginary part. As explained in the introduction, in genus $3$, a dimension count shows that in this case the Schottky problem is trivial. In genus $4$, the ideal  $\mathfrak{I}_4$ is an analytic hypersurface in $\mathfrak{H}_4$. Furthermore, in genus $3$ the second order theta constants verify an equation of degree $16$ which leads to the equation characterizing Jacobians of curves in genus $4$, i.e., Igusa's equation, see \cite[Example 6.2]{vangeemen2016kummer}, \cite{igusa1982irreducibility}.   
An analogous situation can be described when looking at the degenerate theta functions arising from irreducible rational nodal curves. 
The map $\psi$ in \eqref{eq:embeddingR} provides an embedding of the space $\text{Sym}^2(\CC^g)$ in the projective space $\PP^{2^g-1}$.
The dimension count for these spaces is analogous to the one for the spaces involved in \eqref{eq:mapsSchottky}. In particular, for genus $3$, we find that $\overline{im(\psi)} = \overline{\text{Sym}^2(\CC^3)}$ inside $\PP^7$. Note that the image of the map $\psi$ is contained in the locus $V(a_{111}a_{100}a_{010}a_{001} - a_{110}a_{101}a_{011}a_{000})$, where the relation among the $a_i$'s comes from \eqref{eq:THEquartic}. Equality then follows by a direct computation since these are both irreducible varieties of equal dimension. We believe a similar situation should hold for higher genera.
We aim to pursue this direction as a future project.

\section*{Acknowledgments}
We are deeply grateful to Daniele Agostini and Bernd Sturmfels for helpful discussions and support during this project, in addition to feedback about results and the manuscript. We also thank Lara Bossinger, T\"urk\"u \"Ozl\"um \c{C}elik, Marvin Hahn, and Fatemeh Mohammadi for helpful discussions in the starting stages of the project. Finally, we thank the anonymous referees for making suggestions that improved the paper and for pointing us to interesting references. Yelena Mandelshtam was supported by NSF grant DGE 2146752.

\bibliographystyle{elsarticle-harv}

 \end{document}